\documentclass[12pt]{amsart}
\newtheorem{thm}{Theorem}[section]
\newtheorem{lem}[thm]{Lemma}

\theoremstyle{definition}
\newtheorem{defn}[thm]{Definition}

\begin{document}

\markboth{V. E. S. Szab\'o}{Characterization of polynomials}

%

%

\title{CHARACTERIZATION OF POLYNOMIALS}

\author{V. E. S\'ANDOR SZAB\'O}

\address{Department of Analysis Institute of Mathematics\\ Budapest University of Technology and Economics\\
M\H{u}egyetem rkp. 3-9.\\
H-1111, Budapest, 
Hungary\\
sszabo@math.bme.hu}

\maketitle

\begin{abstract}
In 1954 it was proved that if $f$ is infinitely differentiable in the
interval $I$ and some derivative (of order depending on $x$) vanishes
at each $x$, then $f$ is a polynomial. Later it was generalized
for multivariable case. A further extension for distributions is possible. If 
$\Omega\subseteq\mathbf{R}^{n}$ is a non-empty connected open set, 
$u\in\mathcal{D}'(\Omega)$ and for every
$\varphi\in\mathcal{D}(\Omega)$ there exists $m(\varphi)\in\mathbf{N}$
such that $\left(D^{\alpha}u\right)(\varphi)=0$ for all multi-indeces
$\alpha$ satisfying $\left\Vert \alpha\right\Vert =m(\varphi)$,
then $u$ is a polynomial (in distributional sense).
\end{abstract}



\section{Introduction}	

In \cite{2} it was proved that if $f:\mathbf{R}\to\mathbf{R}$,
$f\in C^{\infty}(\mathbf{R})$, and for every $x\in\mathbf{R}$ there
exists $n(x)\in\mathbf{N}$ such that $f^{(n(x))}(x)=0$, then $f$
is a polynomial. Later, see \cite{1}, a similar
result was proved for multi-variable case.

To extend this result for distributions first we introduce some notations
and recall some known results, see e.g. in \cite{4}.

Let $\Omega\subseteq\mathbf{R}^{n}$ be a non-empty open set. In the
discussion of functions of $n$ variables, the term multi-index denotes
an ordered $n$-tuple\[
\alpha=(\alpha_{1},\ldots,\alpha_{n})\]
 of nonnegative integers $\alpha_{i}$ $(i=1,\ldots,n)$. With each
multi-index $\alpha$ is associated the differential operator\[
D^{\alpha}:=\left(\frac{\partial}{\partial x_{1}}\right)^{\alpha_{1}}\cdots\left(\frac{\partial}{\partial x_{n}}\right)^{\alpha_{n}}\]
 whose order is $\left\Vert \alpha\right\Vert :=\alpha_{1}+\ldots+\alpha_{n}$.
If $\left\Vert \alpha\right\Vert =0$ then $D^{\alpha}f:=f$.

We will use also the notation \[D^{k}_{j}:=\left(\frac{\partial}{\partial x_{j}}\right)^{k},
\]
where $j\in\{1,\ldots,n\}$ and $k$ is a nonnegative integer.

The support of a complex function $f$ on any topological space is
the closure of the set $\{x\,|\, f(x)\neq0\}$. 

If $K$ is a compact set in $\Omega$ then $\mathcal{D}_{K}$ denotes
the vector subspace of all complex-valued functions $f\in C^{\infty}(\Omega)$
whose support lies in $K$, $C_{0}^{\infty}(\Omega)$ is the set of
all $f\in C^{\infty}(\Omega)$ whose support is compact and lies in
$\Omega$. It is possible to define a topology on $C^{\infty}(\Omega)$
(generated by the $\left\Vert D^{\alpha}f\right\Vert _{\infty}$ norms)
which makes $C^{\infty}(\Omega)$ into a Fr\'echet space (locally convex
topological vector space whose topology is induced by a complete invariant
metric), such that $\mathcal{D}_{K}$ is a closed subspace of $C^{\infty}(\Omega)$,
$\tau_{K}$ denotes the Fr\'echet space topology of $\mathcal{D}_{K}$.

Choose the non-empty compact sets $K_{i}\subset\Omega$ $(i=1,2,\ldots)$
such that $K_{i}$ lies in the interior of $K_{i+1}$ and $\Omega=\cup K_{i}$,
$\tau_{K_{i}}$ denotes the Fr\'echet space topology of $\mathcal{D}_{K_{i}}$.
Denote $\tau$ the inductive limit topology of $\tau_{K_{i}}$ $(i=1,2,\ldots)$. 

The topological vector space of test functions $\mathcal{D}(\Omega)$
is $C_{0}^{\infty}(\Omega)$ with $\tau$. This topology is independent
of the choice of $K_{i}$ $(i=1,2,\ldots)$. A linear functional on
$\mathcal{D}(\Omega)$ which is continuous with respect to $\tau$
is called a distribution in $\Omega$. The space of all distributions
in $\Omega$ is denoted by $\mathcal{D}'(\Omega)$. 

If $X\subseteq\mathbf{R}^n$ and $Y\subseteq\mathbf{R}^m$ are open sets and $u\in\mathcal{D}'(X)$ and $v\in\mathcal{D}'(Y)$ then their tensor product is $u\otimes v\in\mathcal{D}'(X\times Y)$. (See e.g. \cite{3}, Ch. 4.)

If $W$ is a finite-dimensional subspace of $\mathbf{R}^n$ and $U\subseteq\mathbf{R}^n$ then $\mathrm{proj}_{W} U$ is the orthogonal projection of $U$ on $W$.

In \cite{2} the polynomiality was proved using Baire's
theorem ($\mathbf{R}$ is a complete metric space). In our case the
topology $\tau$ is not locally compact (see \cite{4},
page 17, Theorem 1.22), $\mathcal{D}(\Omega)$ is not metrizable
and not a Baire space (see \cite{4}, page 141, from
last two lines, to page 142, first five lines), so we cannot
apply Baire's theorem to $\mathcal{D}(\Omega)$ immediately. To overcome
this difficulty is not trivial and we need a key lemma, Lemma \ref{lemmacover}.

\section{Lemmas}

To prove our theorem we need some preliminary lemmas. 

\begin{defn}
If $\mathbf{a}=(a_{1},\ldots,a_{n}),\,\mathbf{b}=(b_{1},\ldots,b_{n})\in\mathbf{R}^{n}$
then  $\mathbf{a}<\mathbf{b}$ means $a_{i}<b_{i}$, $i=1,\ldots,n$.
The set $(\mathbf{a},\mathbf{b}):=\{\mathbf{x}\,|\,\mathbf{a}<\mathbf{x}<\mathbf{b}\}$
is a $n$-dimensional open interval.
\end{defn}
\begin{lem}
\label{lemmacover} Suppose $\Gamma$ is an open cover of an open set
$\Omega\subseteq\mathbf{R}^{n}$, and suppose that to each $\omega\in\Gamma$
corresponds a distribution, $\Lambda_{\omega}\in\mathcal{D}'(\omega)$
such that \[
\Lambda_{\omega'}=\Lambda_{\omega''}\quad\mathrm{in}\quad\omega'\cap\omega''\]
 whenever $\omega'\cap\omega''\neq\emptyset$. 

Then there exists a unique $\Lambda\in\mathcal{D}'(\Omega)$ such
that \[
\Lambda=\Lambda_{\omega}\quad\mathrm{in}\quad\omega\]
 for every $\omega\in\Gamma$.
\end{lem}
\begin{proof}
See e.g. \cite{4}, Theorem 6.21.
\end{proof}

\begin{lem}
\label{lemmaonevarpoly} If $\Omega\subseteq\mathbf{R}$ is a non-empty  open interval, $m$ is a non-negative integer and $u\in\mathcal{D}'(\Omega)$ is such that $D^{m}u\equiv 0$ then $u$ is a polynomial (in distributional sense) with degree at most $m-1$.
\end{lem}
\begin{proof}
See \cite{5}, Exercise 7.23, p. 99. In fact the statement was proved for $\Omega=\mathbf{R}$, but the proof gives this more general result.
\end{proof}
\begin{lem}
\label{lemmapartialderiv}
Assume $\Omega\subseteq\mathbf{R}^n$ is a non-empty connected open set and $m$ is a non-negative integer. If $u\in\mathcal{D}'(\Omega)$ then $D^{m}_{n}u\equiv 0$ if and only if
\begin{equation}
u=\sum_{j=0}^{m-1} v_j(x')\otimes p_j(x_n),
\label{pdesol}
\end{equation}
where $x'=(x_1,\ldots,x_{n-1})\in \mathrm{proj}_{\mathbf{R}^{n-1}}\Omega$, $(x',x_n)\in\Omega$, $v_j\in\mathcal{D}'(\mathrm{proj}_{\mathbf{R}^{n-1}}\Omega)$, and $p_j(x_n)$ is a polynomial (in distributional sense) with degree at most $j$.
\end{lem}
\begin{proof}
The special case $\Omega=\mathbf{R}^n$ and $m=1$ was proved in \cite{3}, Theorem 4.3.4, but the proof works for $\Omega=I_1\times\cdots\times I_n$, where $I_i\subseteq\mathbf{R}$, $i=1,\ldots,n$ are open intervals.

It is immediate that (\ref{pdesol}) implies that $D^{m}_{n}u\equiv 0$. To prove the converse we use the method of localization and recovering lemma, 
Lemma \ref{lemmacover}. 

Since open intervals form a base for open sets in $\mathbf{R}^{n}$,
we can write $\Omega=\cup_{\omega\in\Gamma}\omega$,
where $\Gamma$ is an open cover of $\Omega$, and the sets $\omega$ have the form 
$\omega=I_1\times\cdots\times I_n$, where $I_i\subseteq\mathbf{R}$, $i=1,\ldots,n$ are open intervals.
So it is enough to consider the case $\Omega=I_1\times\cdots\times I_n$. In the following, the notation  $c_j$ will be used to designate a real constant.

Since we know the statement for $m=1$ assume that $m=2$. Denote $u_1:=D_n^1 u$. Then $D_n^1 u_1\equiv 0$ and by the case $m=1$ we have 
\[
u_1=v_1\otimes c_1,
\]
where $v_1\in\mathcal{D}'(\mathrm{proj}_{\mathbf{R}^{n-1}}\Omega)$. 
From this equation we obtain
\begin{equation}
D_n^1 u=v_1\otimes c_1.
\label{D1u}
\end{equation}
This equation has a particular solution
\begin{equation}
u_{part}=v_1\otimes(c_1 x_n+c_0). 
\label{u_part}
\end{equation}
Equations (\ref{D1u}) and (\ref{u_part}) yield
\[
D_n^1(u-u_{part})=0.
\label{D1is_zero}
\]
Using again the case $m=1$ we obtain 
\[
u-u_{part}=v_0\otimes \tilde{c}_0,
\]
where 
$v_0\in\mathcal{D}'(\mathrm{proj}_{\mathbf{R}^{n-1}}\Omega)$. It follows
that
\[
u=v_1\otimes(c_1 x+c_0)+v_0\otimes \tilde{c}_0.
\]
Iterating this process we obtain (\ref{pdesol}). 
\end{proof}
\begin{lem}
\label{lemmatotalpartialderiv}
Assume $\Omega\subseteq\mathbf{R}^n$ is a non-empty connected open set and $m$ is a non-negative integer. If $u\in\mathcal{D}'(\Omega)$ then $D^{\alpha}u\equiv 0$ for all multi-indeces $\alpha$ satisfying $\left\Vert \alpha\right\Vert =m$ if and only if $u$ is an $n$-variable polynomial (in distributional sense) with total degree at most $m-1$.

\end{lem}
\begin{proof}
The``if'' part is clear. To prove the ``only if'' part, similarly as in the proof of the previous lemma, it is enough to consider the case $\Omega=I_1\times\cdots\times I_n$. In the following $c_{j,k}$'s will denote arbitrary constant numbers.

By our assumption $D^{m}_{n}u\equiv 0$. Then Lemma \ref{lemmapartialderiv} gives
\begin{equation}
u=\sum_{j=0}^{m-1} v_j(x')\otimes p_j(x_n),
\label{totalpdesol1}
\end{equation}
where $x'=(x_1,\ldots,x_{n-1})\in \mathrm{proj}_{\mathbf{R}^{n-1}}\Omega$, $(x',x_n)\in\Omega$, $v_j\in\mathcal{D}'(\mathrm{proj}_{\mathbf{R}^{n-1}}\Omega)$, and $p_j(x_n)$ is a polynomial (in distributional sense) with degree at most $j$. Since $D_{n-1}^1 D_{n}^{m-1}u\equiv 0$ we get from (\ref{totalpdesol1})
\[
D_{n-1}^1 v_{m-1}(x')\otimes c_{m-1,n}=0,
\]
that is, 
\[
D_{n-1}^1 v_{m-1}(x')=0.
\]
Lemma \ref{lemmapartialderiv} implies
\[
v_{m-1}(x')=v_{m-1,1}(x'')\otimes c_{m-1,n-1},
\]
where $x''=(x_1,\ldots,x_{n-2})\in \mathrm{proj}_{\mathbf{R}^{n-2}}\Omega$, $(x'',x_{n-1})\in\mathrm{proj}_{\mathbf{R}^{n-1}}\Omega$, $v_{m-1,1}\in\mathcal{D}'(\mathrm{proj}_{\mathbf{R}^{n-2}}\Omega)$. Iterating this process, with $D_{i}^1$ $(i=1,\ldots,n-2)$ instead of $D_{n-1}^1$, we obtain
\begin{equation}
v_{m-1}(x')=c_{m-1,1}\otimes c_{m-1,2}\otimes\ldots\otimes c_{m-1,n-1}.
\label{v_m-1}
\end{equation}
Since $D_{i}^1 D_{j}^1 D_{n}^{m-2}u\equiv 0$ $(1\leq i\leq j\leq n-1)$ we get from (\ref{totalpdesol1})
\[
D_{i}^1 D_{j}^1 v_{m-2}(x')\otimes c_{m-2,n}=0,
\]
that is,
\[
D_{i}^1 D_{j}^1 v_{m-2}(x')=0.
\]
Denote $w:= D_{j}^1 v_{m-2}$. Then we have
\[
D_{i}^1 w(x')=0
\]
for all $i\in\{1,\ldots,n-1\}$, 
which implies
\[
D_{j}^1 v_{m-2}=c_{m-2,1}\otimes\ldots\otimes c_{m-2,n-1}.
\]
Similarly as in the proof of Lemma \ref{lemmapartialderiv}, see (\ref{D1u}), we obtain
\begin{eqnarray}
v_{m-2}(x')&=&  (c_{m-2,1} x_1+c_1)\otimes c_{m-2,2}\otimes\ldots\otimes c_{m-2,n-1} \nonumber\\
& +& (c_{m-2,1})\otimes (c_{m-2,2} x_2+c_2)\otimes\ldots\otimes c_{m-2,n-1} \nonumber\\
& +&\ldots\ldots \nonumber\\
& +&c_{m-2,1}\otimes c_{m-2,2}\otimes\ldots\otimes c_{m-2,n-2}\otimes (c_{m-2,n-1}x_n+c_n).
\label{v_m-2}
\end{eqnarray}
Following this method we obtain that $u$ is the sum of tensor products of one variable  polynomials with total degree at most $m-1$. Noticing that the tensor products of one variable  polynomials can be identified with their usual products we obtain the statement of the lemma.
\end{proof}

\section{Main Result}

In the following theorem we assume that $\Omega$ is a connected set, if $\Omega$ had connected components then we could apply our result for each component.

\begin{thm}
If 
$\Omega\subseteq\mathbf{R}^{n}$ is a non-empty connected open set, 
$u\in\mathcal{D}'(\Omega)$ and for every
$\varphi\in\mathcal{D}(\Omega)$ there exists $m(\varphi)\in\mathbf{N}$
such that $\left(D^{\alpha}u\right)(\varphi)=0$ for all multi-indeces
$\alpha$ satisfying $\left\Vert \alpha\right\Vert =m(\varphi)$,
then $u$ is a polynomial (in distributional sense).
\end{thm}

\begin{proof}
In the first step we prove that for each $i=1,2,\ldots$ there exists
a number $\gamma^{(i)}\in\mathbf{N}$ such that $D^{\alpha}u\equiv0$ 
in $\mathcal{D}_{K_{i}}$ for all multi-indeces $\alpha$ satisfying $\left\Vert \alpha\right\Vert =\gamma^{(i)}$
. Denote \[
Z^{(m)}:=\{\varphi\in\mathcal{D}_{K_{i}}\,|\,\left(D^{\alpha}u\right)(\varphi)=0,\mathrm{for\, all}\,\alpha\,\mathrm{satisfying}\left\Vert \alpha\right\Vert =m\},\quad m\in\mathbf{N}.\]
 Obviously\[
\mathcal{D}_{K_{i}}=\bigcup_{m\in\mathbf{N}}Z^{(m)}.\]
 Here $Z^{(m)}$ is closed, because \[
Z^{(m)}=\bigcap_{\left\Vert \alpha\right\Vert =m}\{\varphi\in\mathcal{D}_{K_{i}}\,|\,\left(D^{\alpha}u\right)(\varphi)=0\}\]
 and $D^{\alpha}u$ is continuous. Since $\mathcal{D}_{K_{i}}$ is
a complete metrizable space, Baire's theorem implies that there exists $\gamma^{(i)}\in\mathbf{N}$
such that $\mathrm{int}\, Z^{(\gamma^{(i)})}\neq\emptyset$ ($\mathrm{int}$
is in the topology $\tau_{K_{i}}$). Since $Z^{(\gamma^{(i)})}$ is a
linear subspace in $\mathcal{D}_{K_{i}}$, we obtain $\mathcal{D}_{K_{i}}\equiv Z^{(\gamma^{(i)})}$. 

In the second step we consider the one and multivariable case.\\
If $n=1$ then applying  Lemma \ref{lemmaonevarpoly} the relation  $D^{\gamma^{(i)}}u\equiv0$ implies $u$ is a polynomial
(in distributional sense) in $\mathcal{D}_{K_{i}}$. If $n>1$ then Lemma  \ref{lemmatotalpartialderiv} yields that $u$ is a multivariable polynomial (in distributional sense) in $\mathcal{D}_{K_{i}}$. Since $K_{i}\subset\mathrm{int}\, K_{i+1}$ and
$\mathcal{D}_{K_{i}}\subset\mathcal{D}_{K_{i+1}}$,  by Lemma \ref{lemmacover} we obtain that $u$ is a
polynomial (in distributional sense) in $\mathcal{D}'(\Omega)$.

So the proof of the theorem has been completed.
\end{proof}



\section*{Acknowledgments}

The author thanks the referee for his/her valuable remarks.

\end{document}